\documentclass[12pt]{article}

\usepackage[enableskew]{youngtab}
\usepackage{amsthm}
\usepackage{amssymb}
\usepackage{MnSymbol}

\theoremstyle{plain}
\newtheorem{theorem}{Theorem}[section]
\newtheorem{corollary}[theorem]{Corollary}
\newtheorem{lemma}[theorem]{Lemma}

\newtheorem{proposition}[theorem]{Proposition}

\theoremstyle{definition}
\newtheorem{definition}[theorem]{Definition}
\newtheorem{example}[theorem]{Example}

\DeclareMathOperator{\sgn}{sgn}

\title{Codes and shifted codes}

\author{J. T. Hird\\
\small Department of Mathematics\\[-0.8ex]
\small North Carolina State University, North Carolina, USA\\
\small \texttt{jthird@ncsu.edu}\\
\and
Naihuan Jing\\
\small Department of Mathematics\\[-0.8ex]
\small North Carolina State University, North Carolina, USA\\
\small \texttt{jing@math.ncsu.edu}\\
\and
Ernest Stitzinger\\
\small Department of Mathematics\\[-0.8ex]
\small North Carolina State University, North Carolina, USA\\
\small \texttt{stitz@math.ncsu.edu}\\
}

\date{September 6, 2011}

\begin{document}
\maketitle

\begin{abstract}
The action of the Bernstein operators on Schur functions was given in terms of {\it codes} in \cite{cg} and extended to the analog in Schur Q-functions in \cite{HJS}.  We define a new combinatorial model of
extended codes and show that both of these results follow from a
natural combinatorial relation induced on codes. The new
algebraic structure provides a natural setting for Schur functions indexed by compositions.
%
%
\end{abstract}

\section{Introduction}\label{sec:intro}

In a recent paper, Carrell and Goulden found a formula for the action of $B_n$ on any Schur function in terms of the code of a partition.  In this paper we define a combinatorial model of codes
and show that the commutation relation satisfied by the Bernstein operators induces a natural relation on codes.  We then show that this relation implies Carrell and Goulden's formula as well as a formula for the action Bernstein operators in any order.  This provides a natural generalization
of Schur functions to be indexed by compositions and we use this to prove the analog of Bernstein's theorem in this setting.  We also show the analogous statements for Schur $Q$-functions and the operator $Y_n$ using both codes and {\it shifted codes} of partitions and compare these results to those for Bernstein operators.

Bernstein defined the operators $B(t)$ and $B_n$ on the ring of symmetric functions $\Lambda = \mathbb{C}[p_1, p_2, p_3, \ldots]$ by
$$B(t)=\sum_{n\in \mathbb{Z}}B_nt^n=exp\left(\sum_{k\geq1}\frac{t^k}{k}p_k\right)
exp\left(-\sum_{k\geq1}t^{-k}\frac{\partial}{\partial p_k}\right).
$$
He showed the following two results, often referred to as Bernstein's Theorem \cite{zelevinsky}:
\begin{equation}\label{introcomm}
B_n B_m = - B_{m-1} B_{n+1},
\end{equation}
\begin{equation}\label{introbernthm}
s_\lambda = B_{\lambda_1} B_{\lambda_2} \cdots B_{\lambda_l}.1,
\end{equation}
where $s_\lambda$ is the Schur polynomial indexed by the partition $\lambda = (\lambda_1, \lambda_2, \ldots, \lambda_l)$.  For convenience we also write $B_\mu$ for $B_{\mu_1}B_{\mu_2} \cdots B_{\mu_l}$ for
any composition $\mu $. 

The code of a partition $\lambda$ is defined to be the sequence of letters $R$ and $U$ obtained by tracing right and up along the outside edge of the Young diagram of shape $\lambda$ in the fourth quadrant.  See example \ref{defcode}.  Codes are closely related to Maya diagrams \cite{DJKMO}, one of the oldest combinatorial descriptions of partitions.  Similar structures have also been used in Okounkov's study of random matrices \cite{okounkov}.  Carrell and Goulden showed that
\begin{equation}\label{introcg}
B(t)s_\lambda=\sum_{i\geq1}(-1)^{|\lambda|-|\lambda^{(i)}|+i-1}t^{|\lambda^{(i)}|-|\lambda|}
s_{\lambda^{(i)}},
\end{equation}
where $\lambda^{(i)}$ is a particular partition defined in terms of the code of $\lambda$ \cite{cg}.

We extend the above model of codes to allow a new {\it left} move.
 In the extended model a composition $\mu$ is the sequence of letters $R$, $L$, and $U$ obtained in the same way using the Young diagram of shape $\mu$, including steps right, left, and up.  Using (\ref{introcomm}) we can define an equivalence relation on the codes of compositions by equality of their Bernstein operators.  Using this identification we can find a simple formula for $B_\mu$ using only the code of $\mu$.  In the same spirit the formula (\ref{introcg}) follows easily.

Schur functions can be defined for any composition as indicated by (\ref{introcomm}).
Our new combinatorial
model of codes provides a new explanation for the reason behind this.  Using the aforementioned results we can state the analog of Bernstein's Theorem (\ref{introbernthm}) for Schur functions indexed by compositions.  This shows the relationship between Schur functions and our relation on codes.

In \cite{HJS} we showed the analog of Carrell and Goulden's formula for the vertex operator $Y_n$ defined in \cite{j1} on any Schur $Q$-function both in terms of codes and in terms of {\it shifted codes} of partitions.  We show how these results follow from the relation induced on either codes or shifted codes by the commutation relation satisfied by $Y_n$ and show the similarities between these two approaches and the corresponding approach for Bernstein operators.

\section{Partitions and Code Models}\label{sec:codes}

A partition $\lambda = (\lambda_1, \lambda_2, \ldots, \lambda_l)$ of $n$ is a sequence of nonnegative integers satisfying $\lambda_1 \geq \lambda_2 \geq \ldots \geq \lambda_l$ whose sum is $n$.  A Young diagram of shape $\lambda$ is an array of left aligned boxes with $\lambda_i$ boxes in the $i^{\text{th}}$ row from the top.

\begin{definition}
Define the \textit{code} of a partition $\lambda$ to be the doubly infinite sequence of letters $R$ and $U$ obtained by tracing along the bottom-right edge of the Young diagram of shape $\lambda$ in the fourth quadrant of the $xy$-plane together with the negative $y$- and positive $x$-axes, where $R$ corresponds to a unit right step and $U$ corresponds to a unit up step.
\end{definition}

\begin{example}\label{defcode}
For the partition $\lambda = (4, 2, 2, 1)$, the path described above is shown in bold below.
\begin{center}
\setlength{\unitlength}{.8cm}
\begin{picture}(6,6) 
\linethickness{0.075mm}
\put(0,6){\line(1,0){6}}
\put(0,5){\line(1,0){4}}
\put(0,4){\line(1,0){2}}
\put(0,3){\line(1,0){2}}
\put(0,2){\line(1,0){1}}
\put(0,6){\line(0,-1){6}}
\put(1,6){\line(0,-1){4}}
\put(2,6){\line(0,-1){3}}
\put(3,6){\line(0,-1){1}}
\put(4,6){\line(0,-1){1}}
\linethickness{.5mm}
\put(0,2){\line(0,-1){2}}
\put(0,2){\line(1,0){1}}
\put(1,3){\line(0,-1){1}}
\put(1,3){\line(1,0){1}}
\put(2,5){\line(0,-1){2}}
\put(2,5){\line(1,0){2}}
\put(4,6){\line(0,-1){1}}
\put(4,6){\line(1,0){2}}
\end{picture}
\end{center}
Then the code of $\lambda$ is given by $\alpha = \ldots UUURURUURRURRR\ldots$
\end{example}

We now introduce an extended code model.
\begin{definition}
Define the \textit{code} of a composition $\mu$ to be the doubly infinite sequence of letters $R$, $L$, and $U$ obtained by tracing along the rightmost edge of the Young diagram of shape $\lambda$ in the fourth quadrant of the $xy$-plane together with the negative $y$- and positive $x$-axes, where $R$ corresponds to a unit right step, $L$ corresponds to a unit left step and $U$ corresponds to a unit up step.
\end{definition}

Note that the code of a composition $\mu$ will contain $L$'s exactly when $\mu$ has an exceedance, $\mu_i < \mu_{i+1}$.

\begin{example}
For the composition $\mu = (2, 3, 1, 4)$, the path described above is shown in bold below.
\begin{center}
\setlength{\unitlength}{.8cm}
\begin{picture}(6,6) 
\linethickness{0.075mm}
\put(0,6){\line(1,0){6}}
\put(0,5){\line(1,0){3}}
\put(0,4){\line(1,0){3}}
\put(0,3){\line(1,0){4}}
\put(0,2){\line(1,0){4}}
\put(0,6){\line(0,-1){6}}
\put(1,6){\line(0,-1){4}}
\put(2,6){\line(0,-1){2}}
\put(2,3){\line(0,-1){1}}
\put(3,5){\line(0,-1){1}}
\put(3,3){\line(0,-1){1}}
\put(4,3){\line(0,-1){1}}
\linethickness{.5mm}
\put(0,2){\line(0,-1){2}}
\put(0,2){\line(1,0){4}}
\put(4,3){\line(0,-1){1}}
\put(1,3){\line(1,0){3}}
\put(1,4){\line(0,-1){1}}
\put(1,4){\line(1,0){2}}
\put(3,5){\line(0,-1){1}}
\put(2,5){\line(1,0){1}}
\put(2,6){\line(0,-1){1}}
\put(2,6){\line(1,0){4}}
\end{picture}
\end{center}
Then the code of $\mu$ is given by $\alpha = \ldots UUURRRRULLLURRULURRR\ldots$
\end{example}

We will often write codes multiplicatively.  For instance, we might write $R^4$ rather than $RRRR$ in a code.  As codes in this setting are a special type of word, we also use the terminology \textit{prefix}, \textit{suffix}, and \textit{subword} in the standard way.

The Bernstein operators $B(t)$ and $B_n$ are defined on the ring of symmetric functions $\Lambda=\mathbb C[p_1, p_2, p_3, \ldots]$ by
$$B(t)=\sum_{n\in \mathbb{Z}}B_nt^n=exp\left(\sum_{k\geq1}\frac{t^k}{k}p_k\right)
exp\left(-\sum_{k\geq1}t^{-k}\frac{\partial}{\partial p_k}\right),
$$
where $p_k$ is the $k^{\text{th}}$ power sum symmetric function $p_k = \sum_{i\geq 1} x_i^k$.  Bernstein showed the following two relations for these operators, often referred to as Bernstein's Theorem:
\begin{eqnarray}\label{berncomm}
B_nB_m=-B_{m-1}B_{n+1}
\end{eqnarray}
\begin{eqnarray}\label{bernschur}
s_\lambda = B_{\lambda_1}B_{\lambda_2}\cdots B_{\lambda_l}\cdot1
\end{eqnarray}
where $\lambda = (\lambda_1, \lambda_2, \ldots, \lambda_l)$ is a partition, and $s_\lambda$ is the Schur polynomial indexed by $\lambda$.  For convenience we will write $B_\mu$ for $B_{\mu_1}B_{\mu_2} \cdots B_{\mu_l}$ for any composition $\mu$.

\section{A Relation on Codes}\label{sec:codereln}

Define an equivalence relation $\sim$ on the set of signed codes of compositions by $\alpha \sim \pm\beta$ if and only if $B_{\mu_1}B_{\mu_2} \cdots B_{\mu_l} = \pm B_{\nu_1}B_{\nu_2} \cdots B_{\nu_l}$, and $\alpha \sim 0$ if and only if $B_{\mu_1}B_{\mu_2} \cdots B_{\mu_l} = 0$, where $\alpha$ is the code of $\mu$ and $\beta$ is the code of $\nu$.  This is an equivalence relation since there is a one to one correspondence between a composition $\mu$ and its code $\alpha$.  For convenience of notation, we will write $\alpha_1 \sim \pm\beta_1$ if $\alpha_1$ is a subsequence of $\alpha$ and $\beta_1$ is a subsequence of $\beta$ such that $\alpha_0 \alpha_1 \alpha_2 = \alpha \sim \pm\beta = \pm\beta_0 \beta_1 \beta_2$ where $\alpha_0 = \beta_0$ and $\alpha_2 = \beta_2$.


\begin{proposition}\label{codeprop}
For any positive integer $k$, we have
\begin{eqnarray}
R\left(L^kUR^{k-1}\right) &\sim& \left(L^{k-1}UR^{k-2}\right)R\label{codereln1}\\
U\left(L^kUR^{k-1}\right) &\sim& -\left(L^{k-1}UR^{k-2}\right)U.\label{codereln2}
\end{eqnarray}
\end{proposition}
\begin{proof}
To prove relation (\ref{codereln1}), notice that $LR \sim RL \sim \phi,$ the empty set of no letters.  In other words, any consecutive $L$'s and $R$'s will cancel since they leave the path and the composition unchanged.  The proof is then immediate.  Throughout this paper we will assume that all codes have been reduced, meaning any possible $LR$ or $RL$ cancellations have already been made.

Relation (\ref{codereln2}) is actually a version the commutation relation (\ref{berncomm}) in terms of codes.  Consider the composition $\mu = (n, m),$ where $n < m$.  Using the equation $B_nB_m = - B_{m-1}B_{n+1}$, we get the change in the code of $\mu$ shown below, with the altered path in bold.
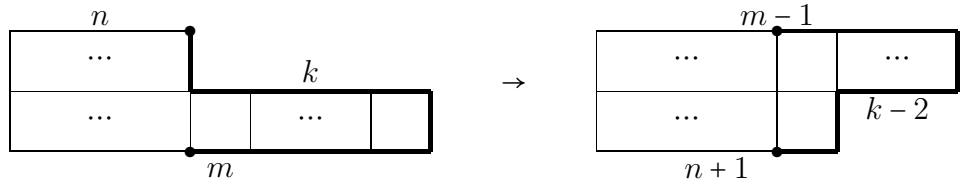
\begin{center}
{\Yvcentermath1 $
\setlength{\unitlength}{.8cm}
\begin{picture}(7,1.5) 
\linethickness{0.075mm}
\put(0,1.1){\makebox(3,1)[b]{$n$}}
\put(0,-2.2){\makebox(7,1)[t]{$m$}}
\put(3,.15){\makebox(4,1)[b]{$k$}}
\put(0,0){\makebox(3,1){$\cdots$}}
\put(0,-1){\makebox(3,1){$\cdots$}}
\put(3,-1){\makebox(4,1){$\cdots$}}
\put(0,1){\line(1,0){3}}
\put(0,0){\line(1,0){7}}
\put(0,-1){\line(1,0){7}}
\put(0,1){\line(0,-1){2}}
\put(3,1){\line(0,-1){2}}
\put(4,0){\line(0,-1){1}}
\put(6,0){\line(0,-1){1}}
\put(7,0){\line(0,-1){1}}
\linethickness{.5mm}
\put(3,1){\line(0,-1){1}}
\put(3,0){\line(1,0){4}}
\put(7,0){\line(0,-1){1}}
\put(3,-1){\line(1,0){4}}
\put(3,1){\circle*{.18}}
\put(3,-1){\circle*{.18}}
\end{picture}
\qquad
\rightarrow
\qquad
\begin{picture}(6,1.5)
\linethickness{0.075mm}
\put(0,1.1){\makebox(6,1)[b]{$m-1$}}
\put(0,-2.1){\makebox(4,1)[t]{$n+1$}}
\put(4,-1.1){\makebox(2,1)[t]{$k-2$}}
\put(0,0){\makebox(3,1){$\cdots$}}
\put(0,-1){\makebox(3,1){$\cdots$}}
\put(4,0){\makebox(2,1){$\cdots$}}
\put(0,1){\line(1,0){6}}
\put(0,0){\line(1,0){6}}
\put(0,-1){\line(1,0){4}}
\put(0,1){\line(0,-1){2}}
\put(3,1){\line(0,-1){2}}
\put(4,1){\line(0,-1){2}}
\put(6,1){\line(0,-1){1}}
\linethickness{.5mm}
\put(3,1){\line(1,0){3}}
\put(6,1){\line(0,-1){1}}
\put(4,0){\line(1,0){2}}
\put(4,0){\line(0,-1){1}}
\put(3,-1){\line(1,0){1}}
\put(3,1){\circle*{.18}}
\put(3,-1){\circle*{.18}}
\end{picture}
$}
\end{center}

\vspace{1cm} 

Since the bold path above begins and ends at the same point at the same point horizontally, and since the codes of these two compositions are related from the definition of $\sim$, we have that
\begin{equation}\label{codereln0}
R^kUL^kU \sim -RUR^{k-2}UL^{k-1}.
\end{equation}


Now multiply both sides of relation (\ref{codereln0}) by $L^k$ on the left and by $R^{k-1}$ on the right and again use the fact that $LR \sim RL \sim \phi$ to obtain (\ref{codereln2}).
If we insert the prefix $L^k$ before both of the subsequences in (\ref{codereln0}) and cancel any consecutive $LR$'s, we get exactly identity (\ref{codereln2}).
\end{proof}
We can also prove relation (\ref{codereln2}) directly from the commutation identity of the Bernstein functions (\ref{berncomm}) by following the bold line below.
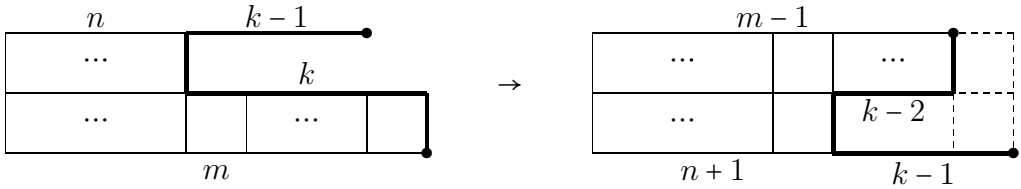
\begin{center}
{\Yvcentermath1 $
\setlength{\unitlength}{.8cm}
\begin{picture}(7,1.5) 
\linethickness{0.075mm}
\put(0,1.1){\makebox(3,1)[b]{$n$}}
\put(0,-2.2){\makebox(7,1)[t]{$m$}}
\put(3,.15){\makebox(4,1)[b]{$k$}}
\put(3,1.1){\makebox(3,1)[b]{$k-1$}}
\put(0,0){\makebox(3,1){$\cdots$}}
\put(0,-1){\makebox(3,1){$\cdots$}}
\put(3,-1){\makebox(4,1){$\cdots$}}
\put(0,1){\line(1,0){3}}
\put(0,0){\line(1,0){7}}
\put(0,-1){\line(1,0){7}}
\put(0,1){\line(0,-1){2}}
\put(3,1){\line(0,-1){2}}
\put(4,0){\line(0,-1){1}}
\put(6,0){\line(0,-1){1}}
\put(7,0){\line(0,-1){1}}
\linethickness{.5mm}
\put(3,1){\line(1,0){3}}
\put(3,1){\line(0,-1){1}}
\put(3,0){\line(1,0){4}}
\put(7,0){\line(0,-1){1}}
\put(6,1){\circle*{.18}}
\put(7,-1){\circle*{.18}}
\end{picture}
\qquad
\rightarrow
\qquad
\begin{picture}(7,1.5)
\linethickness{0.075mm}
\put(0,1.1){\makebox(6,1)[b]{$m-1$}}
\put(0,-2.1){\makebox(4,1)[t]{$n+1$}}
\put(4,-1.1){\makebox(2,1)[t]{$k-2$}}
\put(4,-2.15){\makebox(3,1)[t]{$k-1$}}
\put(0,0){\makebox(3,1){$\cdots$}}
\put(0,-1){\makebox(3,1){$\cdots$}}
\put(4,0){\makebox(2,1){$\cdots$}}
\put(0,1){\line(1,0){6}}
\put(0,0){\line(1,0){6}}
\put(0,-1){\line(1,0){4}}
\put(0,1){\line(0,-1){2}}
\put(3,1){\line(0,-1){2}}
\put(4,1){\line(0,-1){2}}
\put(6,1){\line(0,-1){1}}
\multiput(7,1)(0,-.2){10}{\line(0,-1){.1}} 
\multiput(7,1)(-.2,0){5}{\line(-1,0){.1}} 
\multiput(7,0)(-.2,0){5}{\line(-1,0){.1}} 
\multiput(6,0)(0,-.2){5}{\line(0,-1){.1}} 
\linethickness{.5mm}
\put(6,1){\line(0,-1){1}}
\put(4,0){\line(1,0){2}}
\put(4,0){\line(0,-1){1}}
\put(4,-1){\line(1,0){3}}
\put(6,1){\circle*{.18}}
\put(7,-1){\circle*{.18}}
\end{picture}
$}
\end{center}

\vspace{1cm}

In this construction our relation on codes in (\ref{codereln2}) gives exactly the path along the rightmost edge of the diagram for $\mu = (n,m)$ from the bottom-right corner of the bottom row to the point two units up and one unit left of the starting position.  The subsequence on the right-hand side of (\ref{codereln2}) is exactly the path along the rightmost edge of the diagram for $\nu = (m-1, n+1)$ which starts and ends at those same points.  So relation (\ref{codereln2}) follows directly from the fact that the Bernstein functions indexed by these two compositions are related by (\ref{berncomm}).

This interpretation is less intuitive than the original construction, but it serves to show the deep connection that still exists between this relation on codes and the commutation relation of the Bernstein functions.

Notice that the special case $k=1$ in (\ref{codereln2}) gives that
$$ULU = U(L^1UR^0) \sim -(L^0UR^{-1})U = -ULU,$$
so $ULU \sim 0,$ since any composition whose code contains this subword must contain a subsequence $(n+1, n)$ and $B_{n+1}B_n = -B_nB_{n+1} = 0$.

\section{Codes of Compositions}\label{sec:compos}

We now want to use the relation on codes of compositions we developed in Section \ref{sec:codereln} to study $B_\mu$ where $\mu$ is a composition.

\begin{lemma}\label{SameSum}
Suppose that the codes $\alpha$ and $\beta$ of two compositions $\mu$ and $\nu$ differ only by one of the relations (\ref{codereln1}) or (\ref{codereln2}).  Then $\mu$ and $\nu$ have the same number of components, $l$, and the same sum, $\mu_1+\mu_2+\cdots +\mu_l = \nu_1+\nu_2+\cdots +\nu_l$.
\end{lemma}
\begin{proof}
If $\alpha$ and $\beta$ differ by relation (\ref{codereln1}), notice the two sides of (\ref{codereln1}) are two different descriptions of the same path, $L^{k-1}UR^{k-1}$, so we actually have that $\mu$ = $\nu$.

If $\alpha$ and $\beta$ differ by relation (\ref{codereln1}), notice that both sides of (\ref{codereln2}) have a net shift of one unit leftward and two units upward.  This implies that $\mu$ and $\nu$ are the same composition except for the two components determined by the two upward steps in the changed subword.  In other words, both $\mu$ and $\nu$ have length $l$ and $\mu_i = \nu_i$ for $i = 1, 2, \ldots, j-1, j+2, \ldots, l$ for some $1\leq j \leq l-1$.

It remains only to show that the two components of $\mu$ that are changed to get $\nu$ have the same sum.  Notice that this case corresponds exactly to the picture above, so the corresponding components of $\mu$ and $\nu$ are $(n,m)$ and $(m-1,n+1)$ for some integer $n$ with $m=n+k$.
\end{proof}

\begin{theorem}\label{BetaK}
Let $\mu$ be any composition of $m$ with code $\alpha$.  Suppose that $\alpha$ can be written in the form
$$\alpha = \ldots \beta_3\beta_2\beta_1L^kU \gamma_1 \gamma_2 \gamma_3 \ldots$$
where $\beta = \ldots \beta_3\beta_2\beta_1$ consists only of $R$'s and $U$'s and $\beta_1=U$.

\begin{itemize}
\item If $\beta_k = U$, then $B_\mu = 0$.
\item If $\beta_k = R$, then $B_\mu = (-1)^j B_\nu$, where $j$ is the number of $U$'s in $\beta_{k-1}\ldots \beta_2\beta_1$ and $\nu$ is the composition of $m$ with code given by
$$\ldots \beta_{k+1}U\beta_{k-1}\ldots \beta_3\beta_2\beta_1L^{k-1} \gamma_1 \gamma_2 \gamma_3 \ldots$$
\end{itemize}
\end{theorem}

\begin{proof}
First, rewrite $\alpha$ in the form:
\begin{eqnarray*}
\alpha &=& \ldots\beta_3\beta_2\beta_1L^kU\gamma_1 \gamma_2 \gamma_3\ldots \\
&\sim& \ldots\beta_3\beta_2\beta_1\left(L^kUR^{k-1}\right)L^{k-1}\gamma_1 \gamma_2 \gamma_3\ldots
\end{eqnarray*}
By Proposition \ref{codeprop}, every time we permute $\left(L^kUR^{k-1}\right)$ left past a letter of $\beta$, $k$ decreases by one, and the sign changes if that letter was a $U$.  Thus if we permute $\left(L^kUR^{k-1}\right)$ past $k-1$ letters, we get that
\begin{eqnarray*}
\alpha &\sim& (-1)^j \ldots\beta_{k+1}\beta_k\left(L^1UR^0\right)\beta_{k-1}\ldots \beta_2\beta_1L^{k-1}\gamma_1 \gamma_2 \gamma_3\ldots \\
&=& (-1)^j \ldots\beta_{k+1}(\beta_kLU)\beta_{k-1}\ldots \beta_2\beta_1L^{k-1}\gamma_1 \gamma_2 \gamma_3\ldots
\end{eqnarray*}
where $j$ is the number of $U$'s in $\beta_{k-1}\ldots \beta_2\beta_1$.  If $\beta_k=U$, then $\alpha$ is related to a code with the subword $\beta_kLU=ULU \sim 0$, thus $B_\mu = 0$.
If $\beta_k =R$, then $\beta_kLU=RLU\sim U$, so
\begin{equation}\label{AlphSim}
\alpha \sim (-1)^j \ldots\beta_{k+1}U\beta_{k-1}\ldots \beta_2\beta_1L^{k-1}\gamma_1 \gamma_2 \gamma_3\ldots
\end{equation}
Thus $B_\mu = (-1)^j B_\nu$, where the code of $\nu$ is given by the right hand side of equation (\ref{AlphSim}).  Since in each step we used only relation (\ref{codereln2}), by Lemma \ref{SameSum} $\nu$ is also a composition of $m$.
\end{proof}

Another way to understand this theorem is to notice that the code of $\nu$ is obtained from the code $\alpha$ of $\mu$ by replacing the letter $\beta_k=R$ which is $k$ positions left of the leftmost $L$ in $\alpha$ with $U$ and by replacing $L^kU$ with $L^{k-1}$, letting $j$ be the number of $U$'s between these two positions.

\begin{corollary}\label{CGLem}
Let $\lambda = (\lambda_1, \lambda_2, \ldots, \lambda_l)$ be a partition and $n$ be any integer with $n < \lambda_1$.  Let $k = \lambda_1 - n$, and let $\zeta$ be the letter $k-1$ positions left of the rightmost $U$ in the code of $\lambda$.
\begin{itemize}
\item If $\zeta=U$, then $B_nB_\lambda = 0$.
\item If $\zeta = R$, then $B_nB_\lambda = (-1)^{j+1} B_\nu$, where $j$ is the number of $U$'s between the rightmost $U$ and $\zeta$, and where $\nu$ is the partition whose code is given by replacing $\zeta$ by $U$.
\end{itemize}
\end{corollary}

\begin{proof}
The corollary follows immediately from Theorem \ref{BetaK}, for the special case where $\mu = (n, \lambda_1, \lambda_2, \ldots, \lambda_l)$.  In this case, $\beta_1$ is the rightmost $U$ in the code of $\lambda$, $\beta_k = \zeta$, $\gamma = \gamma_1 \gamma_2 \gamma_3 \ldots = RRR\ldots$ and $\beta\gamma = \ldots \beta_3\beta_2\beta_1 \gamma_1 \gamma_2 \gamma_3 \ldots$ is exactly the code of $\lambda$.  In this case, when we replace $L^kU$ by $L^{k-1}$, it cancels with the first $k-1$ $R$'s in $\gamma$, so our result will actually be a partition.
\end{proof}

\begin{corollary}\label{B_comp}
Let $\mu$ be any composition of $m$ with code $\alpha$.  Then either $B_\mu = 0$ or $B_\mu = \pm B_\lambda$ for some partition $\lambda$ of $m$ with the same length as $\mu$.
\end{corollary}

The proof of this corollary follows directly and immediately from Theorem \ref{BetaK} by using induction on the number of $L$'s in $\alpha$.  However, we will present the proof using induction on the number of $U$'s to the right of the leftmost $L$ in $\alpha$ to better generalize to our later results.

\begin{proof}
Consider the leftmost sequence of consecutive $L$'s in $\alpha$.  If there are no $U$'s right of this sequence, then all the $L$'s cancel, and $\mu$ is a partition.  So assume there is at least one $U$ after the leftmost $L$ and assume $\alpha$ is written in reduced form.  Then in the notation of the theorem, we have
$$\alpha = \ldots \beta_3\beta_2\beta_1L^kU\gamma_1\gamma_2\gamma_3\ldots$$
where $\beta = \ldots \beta_3\beta_2\beta_1$ consists only of $R$'s and $U$'s.  By Theorem \ref{BetaK}, either $\alpha \sim 0$ or
$$\alpha = \ldots \beta_{k+1}U\beta_{k-1}\ldots \beta_3\beta_2\beta_1L^{k-1}\gamma_1\gamma_2\gamma_3\ldots$$
In the later case the number of $U$'s to the right of the leftmost $L$ has decreased by one (or more if $L^{k-1}$ cancels completely with $R$'s in $\gamma$) so the result holds by induction.\end{proof}

In particular, this corollary provides a simple way to compute $B_\mu \cdot 1$ in terms of Schur functions for any composition $\mu$, since $B_\lambda\cdot 1 = s_\lambda$ by equation (\ref{bernschur}) for any partition $\lambda$.

In fact, given any composition $\mu$, the number of times we have to apply Theorem \ref{BetaK} to get  $B_\mu = \pm B_\lambda$ for some partition $\lambda$ is less than or equal to the number of $U$'s right of the leftmost $L$ in the code of $\mu$.  Equivalently, the maximum number of steps is the largest $i$ such that $\mu_i<\mu_{i+1}$, i.e. the position of the last exceedance in $\mu$.  From the remark before the proof, we can also say that the number of steps is less than or equal to the total number of $L$'s in the code of $\mu$.

We now present a purely combinatorial approach to computing the partition $\lambda$ such that $B_\mu = \pm B_\lambda$ for a given composition $\mu$, as well as the sign itself.

\begin{proposition}
Given a composition $\mu$ with code $\alpha$, $B_\mu = (-1)^jB_\lambda$, where $j$ and the code of $\lambda$ are obtained by reading then deleting letters left to right starting with the leftmost $L$ in $\alpha$ and keeping track of a position in the code, starting with the same $L$.
\begin{itemize}
\item Every time an $L$ or $R$ is read, move one position in that direction.
\item If a $U$ is read and the letter in the current position is a $U$, $B_\mu = 0$.
\item If a $U$ is read and the letter in the current position is an $R$, increase $j$ by the number of $U$'s between these two positions and replace the $R$ with a $U$.  Move right one position.
\end{itemize}
Stop when the current position indicates the next letter to be read.  If there are still any $L$'s in the code, repeat this process.
\end{proposition}

Note that in this method you must keep track of two positions, the position which is being read and the ``current position" which indicates which letter will be changed by any $U$'s which are read.  Note also that this proposition holds even if $\alpha$ is not in reduced form.

\begin{proof}
Suppose that the leftmost sequence of consecutive $L$'s in $\alpha$ is $L^k$ and that this is followed by a $U$ (this will always happen if $\alpha$ is in reduced form).  Then by Theorem \ref{BetaK}, $\alpha \sim 0$ if the letter $k$ positions left of $L^k$ is a $U$ and if this letter is an $R$, replace it by $U$, replace $L^kU$ by $L^{k-1}$, and change the sign by the number of $U$'s between these positions.  Following the algorithm in the proposition, we read $k$ $L$'s so we move left $k$ positions.  Since the next letter is a $U$, we perform the same change to the sign and the letter in the current position and will next consider the letter one position to the right of the changed position, which corresponds to $L^kU$ being replaced by $L^{k-1}$.

If the letter after $L^k$ is $R$, then $L^kR \sim L^{k-1}$ so we will next consider the letter one position further to the right.  If the letter after $L^k$ is $L$, then $L^kL \sim L^{k+1}$ so we will next consider the letter one position further to the left.  We stop this process when all the $L$'s in $L^k$ have been cancelled.
\end{proof}

\begin{example}
Consider the composition $\mu = (1,3,1,6,2)$.  Then the proposition gives the following.  In each step below we read off letters until a $U$ is reached and underline these letters, while the arrows indicate how the current position changes.  Note that we delete the underlined letters after each step.

\begin{center}
$
\begin{array}{cccl}
\alpha & = &  & \ldots UUURRURRRRU\underline{LLLLLU}RRULLURRRR\ldots \\
&&& \phantom{\ldots UUURRU}
\raisebox{1pt}{$\,\downarrow$}
\raisebox{-3pt}{$\! \curvearrowright$}
\raisebox{5pt}{$\!\!\!\!\!\! \lcurvearrowleft\hspace{-1pt}\lcurvearrowleft\hspace{-1pt}\lcurvearrowleft\hspace{-1pt}\lcurvearrowleft\hspace{-1pt}\lcurvearrowleft$}\\
& \rightarrow & (-1)^{1\phantom{+1+2}} & \ldots UUURRUURRRU\underline{RRU}LLURRRR\ldots \\
&&& \phantom{\ldots UUURRUU}
\raisebox{5pt}{$\,\,\rcurvearrowright\hspace{-1pt}\rcurvearrowright$}
\raisebox{1pt}{$\!\downarrow$}
\raisebox{-3pt}{$\!\curvearrowright$} \\
& \rightarrow & (-1)^{1+1\phantom{+2}} & \ldots UUURRUURRUU\underline{LLU}RRRR\ldots \\
&&& \phantom{\ldots UUURRUUR}
\raisebox{1pt}{$\,\downarrow$}
\raisebox{-3pt}{$\! \curvearrowright$}
\raisebox{5pt}{$\!\!\!\!\!\! \lcurvearrowleft\hspace{-1pt}\lcurvearrowleft$}\\
& \rightarrow & (-1)^{1+1+2} & \ldots UUURRUURUUU\underline{RR}RR\ldots \\
&&& \phantom{\ldots UUURRUURU}
\raisebox{5pt}{$\,\,\rcurvearrowright\hspace{-1pt}\rcurvearrowright$}\\
\alpha & \sim & (-1)^{1+1+2} & \ldots UUURRUURUUURR\ldots \\
\end{array}
$
\end{center}
We stop since the current position is the same as the next letter to be read.  This final code has no remaining $L$'s so we are done.  This code is the code of the partition $\lambda = (3,3,3,2,2)$, so $B_{(1,3,1,6,2)} = (-1)^4 B_{(3,3,3,2,2)} = B_{(3,3,3,2,2)}$.  In particular, this means that $B_1B_3B_2B_6B_3\cdot 1 = B_\mu \cdot 1 = B_\lambda \cdot 1 = s_\lambda = s_{(3,3,3,3,3)}$.  
\end{example}

\begin{definition}
Define $r_i(\lambda)$ to be the number of $R$'s in the code of the partition $\lambda$ left of the $i^{\text{th}}$ $U$ from the right in the code of $\lambda$.
\end{definition}

Note that $r_i(\lambda) = \lambda_i$, the $i^{\text{th}}$ component of $\lambda$.

\begin{definition}\label{(i)}
For any partition $\lambda$, define $\lambda^{(i)}$ to be the partition obtained from the code of $\lambda$ by replacing the $i^{\text{th}}$ $R$ from the left in the code of $\lambda$ by $U$.
\end{definition}

In particular, this means that $\lambda^{(i)} = (\lambda_1-1, \lambda_2-1, \ldots, \lambda_j-1, i-1, \lambda_{j+1}, \ldots, \lambda_l)$, where $\lambda_j-1 \geq i-1 \geq \lambda_{j+1}$, so $\lambda_j \geq i > \lambda_{j+1}$.  By convention, we take $\lambda_0=\infty$ and $\lambda_{l+1} = 0$, so the formula for $\lambda^{(i)}$ holds for $0\leq j \leq l$.

\begin{theorem}\label{CGCodeThm}  For any partition $\lambda$,
\begin{equation}\label{cgintro}
B(t)s_\lambda=\sum_{i\geq1}(-1)^{|\lambda|-|\lambda^{(i)}|+i-1}t^{|\lambda^{(i)}|-|\lambda|}
s_{\lambda^{(i)}}.
\end{equation}
\end{theorem}
\begin{proof}
From the definition of $B(t) = \sum_{n\in \mathbb{Z}} B_nt^n$, and equation (\ref{bernschur}), we have that $B(t)s_\lambda = \sum_{n\in\mathbb{Z}}B_nt^n(B_{\lambda_1}B_{\lambda_2}\cdots B_{\lambda_l}\cdot 1)$.  So it is sufficient to consider the coefficient of $t^n$, $B_nB_{\lambda_1}B_{\lambda_2}\cdots B_{\lambda_l}$.

If $n \geq \lambda_1$, then $B_nB_\lambda = B_{(n,\lambda)}$ is already in decreasing order.  In fact, since $(n, \lambda) = \lambda^{(n+1)}$, we can write $B_nB_\lambda \cdot 1= B_{(n,\lambda)} \cdot 1= B_{\lambda^{(n+1)}} \cdot 1 = s_{\lambda^{(n+1)}}$.

If $n < \lambda_1$, then we write the code of $\lambda$ in the form
$$\alpha = \ldots \beta_3 \beta_2 \beta_1 RRR\ldots$$
where $\beta_1 = U$.  Let $j$ be the number of $U$'s in $\beta_k \beta_{k-1} \ldots \beta_2 \beta_1$, where $k = \lambda_1 - n$.

From Corollary \ref{CGLem}, if $\beta_k=U$, then $\alpha \sim 0$, so $B_nB_\lambda = 0$.  Since $\beta_k$ is the $j^{\text{th}}$ $U$ from the right in $\beta_k \beta_{k-1} \ldots \beta_2 \beta_1$, $\beta_k$ is also the $j^{\text{th}}$ $U$ from the right in $\alpha$.  Since $r_i(\lambda) = \lambda_i$ is the number of $R$'s left of the $i^{\text{th}}$ $U$ from the right in $\alpha$, the number of $R$'s between the $(i+1)^{\text{th}}$ $U$ from the right and the $i^{\text{th}}$ $U$ from the right is $r_i(\lambda) - r_{i+1}(\lambda) = \lambda_i - \lambda_{i+1}$.  In this case we can write
$$\alpha = \ldots \beta_{k+1} U R^{\lambda_{j-1}-\lambda_j}U \ldots UR^{\lambda_2-\lambda_3} UR^{\lambda_1-\lambda_2} URRR\ldots$$
So the $k$ letters $\beta_k \beta_{k-1} \ldots \beta_2 \beta_1$ consist of $j$ $U$'s and $(\lambda_1-\lambda_2) + (\lambda_2-\lambda_3) + \cdots + (\lambda_{j-1}-\lambda_j)$ $R$'s.  So we have:
\begin{eqnarray*}
\lambda_1 - n \quad = \quad k &=& j + (\lambda_1-\lambda_2) + (\lambda_2-\lambda_3) + \cdots + (\lambda_{j-1}-\lambda_j)\\
n+j &=& \lambda_1 - (\lambda_1-\lambda_2) - (\lambda_2 - \lambda_3) - \cdots - (\lambda_{j-1} - \lambda_j)\\
n+j &=& \lambda_j\\
n &=& \lambda_j - j
\end{eqnarray*}
Thus $B_nB_\lambda = 0$ exactly when $n = \lambda_j - j$ for some $j = 1, 2, \ldots, l$.

If $n \neq \lambda_j - j$ for any $j$,  then we must have that $\beta_k = R$.  Hence by Corollary \ref{CGLem}, $B_{(n,\lambda)} = (-1)^j B_\nu$, where $j$ is the number of $U$'s in $\beta_k \beta_{k-1} \ldots \beta_2 \beta_1$ and $\nu$ is the partition whose code is the same as $\alpha$ except that $\beta_k = U$.  Then the number of $R$'s in $\beta_k \beta_{k-1} \ldots \beta_2 \beta_1$ is $k-j$, and the number of $R$'s in all of $\beta = \ldots \beta_{k+1} \beta_k \beta_{k-1} \ldots \beta_2 \beta_1$ is $\lambda_1 = r_1(\lambda)$, which is the number of $R$'s left of the rightmost $U$ in $\alpha$, $\beta_1$.  Thus there are $\lambda_1 - (k-j) = \lambda_1 - k + j = \lambda_1 - (\lambda_1 - n) + j = n+j$ $R$'s left of $\beta_k$ in $\alpha$.  So $\nu$ is the partition obtained by replacing the $(n+j+1)^{\text{th}}$ $R$ from the left in the code $\alpha$ of $\lambda$ with a $U$.  Thus $\nu = \lambda^{(n+j+1)}$ and $B_nB_\lambda \cdot 1 = (-1)^j B_\nu \cdot 1= (-1)^j s_{\lambda^{(n+j+1)}}$.

Using the convention that $\lambda_0=\infty$ and $\lambda_{l+1} = 0$, we obtain a cover of the integers greater than or equal to $-l$: $[-l, \infty)=\displaystyle\bigcup_{j=0}^l (\lambda_{j+1}-(j+1), \lambda_j-j]$.  We know that if $n < -l$, then $B_ns_\lambda = (-1)^l B_{\lambda_1-1}B_{\lambda_2-1}\cdots B_{\lambda_l-1}B_{n+l} \cdot 1 = 0$ since $B_{-m}\cdot 1 = B_{-m}B_0\cdot 1 =B_{-1}B_{-1}\cdots B_{-1}B_0\cdot 1= 0$.  Also note that the right limits of this cover are of the form $n = \lambda_j - j$, so we know that $B_nB_\lambda = 0$.  Thus:
$$B(t)s_\lambda
= \sum_{n\in \mathbb{Z}} t^nB_ns_\lambda
= \sum_{j=0}^l \sum_{n=\lambda_{j+1}-j}^{\lambda_j-j} t^nB_ns_\lambda
= \sum_{j=0}^l \sum_{n=\lambda_{j+1}-j}^{\lambda_j-j-1} (-1)^j t^n s_{\lambda^{(n+j+1)}}.
$$
But looking at the summation on the right hand side, $\lambda_{j+1}-j \leq n \leq \lambda_j-j-1$, so $\lambda_{j+1}+1\leq n+j+1 \leq \lambda_j$, so $\lambda_{j+1} < n+j+1 \leq \lambda_j$ so the indices of $\lambda^{(n+j+1)}$ in the summation cover all positive integers.  Hence
$$B(t)s_\lambda
= \sum_{j=0}^l \sum_{n=\lambda_{j+1}-j}^{\lambda_j-j-1} (-1)^j t^n s_{\lambda^{(n+j+1)}}
=\sum_{i\geq1}(-1)^{|\lambda|-|\lambda^{(i)}|+i-1}t^{|\lambda^{(i)}|-|\lambda|}
s_{\lambda^{(i)}},$$
since $|\lambda|-|\lambda^{(i)}|+i-1 = j$ and $|\lambda^{(i)}|-|\lambda| = n$ from the definition of $\lambda^{(i)}$.
\end{proof}

\section{Schur Functions indexed by compositions}\label{sec:genschur}

In this section we will use the functions, {\it Schur functions indexed by compositions}, to show the usefulness of the results obtained in the previous section.  In this section we will consider only compositions consisting of all nonnegative components.

Let $\delta = (l-1, l-2, \ldots, 2, 1, 0)$ and define $a_\mu = \det ({x_i}^{\mu_j})_{1\leq i, j \leq l}$ for any composition $\mu$ of $n$ of length $l$.  Then one classical definition of the {\it Schur polynomials} is given by $s_\lambda (x_1, x_2, \ldots, x_l) = \dfrac{a_{\lambda+\delta}}{a_\delta}$ for any partition $\lambda$ of $n$ of length $l$.  This definition can also be generalized to compositions as follows.

\begin{definition}
Let $\mu$ be a composition of $n$ of length $l$.  Define the {\it Schur polynomial indexed by the composition $\mu$} to be $s_\mu(x_1, x_2, \ldots, x_l) = \dfrac{a_{\mu+\delta}}{a_\delta}$.
\end{definition}

\begin{lemma}\label{SchurLem}
Let $\mu$ and $\nu$ be any two compositions of length $l$ with codes $\alpha$ and $\beta$.  Then $s_\mu(x_1, x_2, \ldots, x_l) = \pm s_\nu(x_1, x_2, \ldots, x_l)$ if and only if $\alpha \sim \pm \beta$ as described in Sections \ref{sec:codereln} and \ref{sec:compos}.
\end{lemma}

\begin{proof}
By definition, $s_\mu(x_1, x_2, \ldots, x_l) = \dfrac{a_{\mu+\delta}}{a_\delta}$ and $s_\nu(x_1, x_2, \ldots, x_l) = \dfrac{a_{\nu+\delta}}{a_\delta}$, so it suffices to show that $a_{\mu+\delta} = a_{\nu+\delta}$ precisely when the codes of $\mu$ and $\nu$ are related.  By the antisymmetry of the determinant, we have
$$\begin{array}{rcl}
a_{\mu+\delta} & = & \phantom{- \,\,\,}
\left|\begin{array}{cccccc}
{x_1}^{\mu_1+l-1} & \cdots & {x_1}^{\mu_i + l - i} & {x_1}^{\mu_{i+1} +l - i - 1} & \cdots & {x_1}^{\mu_l + 0} \\
{x_2}^{\mu_1+l-1} & \cdots & {x_2}^{\mu_i + l - i} & {x_2}^{\mu_{i+1} +l - i - 1} & \cdots & {x_2}^{\mu_l + 0} \\
\vdots && \vdots & \vdots && \vdots \\
{x_l}^{\mu_1+l-1} & \cdots & {x_l}^{\mu_i + l - i} & {x_l}^{\mu_{i+1} +l - i - 1} & \cdots & {x_l}^{\mu_l + 0} \\
\end{array}\right| \\
\\
& = & - \,\,\,
\left|\begin{array}{cccccc}
{x_1}^{\mu_1+l-1} & \cdots & {x_1}^{\mu_{i+1} +l - i - 1} & {x_1}^{\mu_i + l - i} & \cdots & {x_1}^{\mu_l + 0} \\
{x_2}^{\mu_1+l-1} & \cdots & {x_2}^{\mu_{i+1} +l - i - 1} & {x_2}^{\mu_i + l - i} & \cdots & {x_2}^{\mu_l + 0} \\
\vdots && \vdots & \vdots && \vdots \\
{x_l}^{\mu_1+l-1} & \cdots & {x_l}^{\mu_{i+1} +l - i - 1} & {x_l}^{\mu_i + l - i} & \cdots & {x_l}^{\mu_l + 0} \\
\end{array}\right| \\
\\
& = & - \,\,\,
\left|\begin{array}{cccccc}
{x_1}^{\mu_1+l-1} & \cdots & {x_1}^{(\mu_{i+1}-1) +l - i} & {x_1}^{(\mu_i+1) + l - i - 1} & \cdots & {x_1}^{\mu_l + 0} \\
{x_2}^{\mu_1+l-1} & \cdots & {x_2}^{(\mu_{i+1}-1) +l - i} & {x_2}^{(\mu_i+1) + l - i - 1} & \cdots & {x_2}^{\mu_l + 0} \\
\vdots && \vdots & \vdots && \vdots \\
{x_l}^{\mu_1+l-1} & \cdots & {x_l}^{(\mu_{i+1}-1) +l - i - 1} & {x_l}^{(\mu_i+1) + l - i - 1} & \cdots & {x_l}^{\mu_l + 0} \\
\end{array}\right| \\
& = & - \,\,\, a_{(\mu_1, \ldots, \mu_{i-1}, \mu_{i+1} - 1, \mu_i + 1, \mu_{i+2}, \ldots, \mu_l) + \delta} .
\end{array}
$$
Hence, the indices, $\mu$, of $a_{\mu+\delta}$ satisfy the same commutation relation as the indices of the Bernstein operators in equation \ref{berncomm}.  Thus, by the definition of the equivalence relation $\sim$ on signed codes of compositions, $a_{\mu+\delta} = a_{\nu+\delta}$ if and only if the codes of $\mu$ and $\nu$ are related, so the result holds.
\end{proof}

\begin{definition}
Let $\mu$ be a composition of $n$ of length $l$.  Define the {\it Schur function indexed by the composition $\mu$}, $s_\mu$, to be the unique symmetric function in $\displaystyle\oplus_{k=0}^l \Lambda_k$ whose restriction to $x_{l+1} = x_{l+2} = \cdots = 0$ is the Schur polynomial $s_\mu(x_1, x_2, \ldots, x_l)$.
\end{definition}

Note that if $\mu = \lambda$ is a partition, we obtain recover the classical definition of Schur functions.  Note also that in $\Lambda$ there is not a unique symmetric function whose restriction to $x_{l+1} = x_{l+2} = \cdots = 0$ is $s_\mu(x_1, x_2, \ldots, x_l)$.  In fact, given any such function $s_\mu$, any element of the coset $s_\mu + <e_{l+1}, e_{l+2}, \ldots >$ will satisfy this condition, where $e_m = \sum_{1\leq i_1 < i_2 < \cdots < i_m} (x_{i_1}x_{i_2} \cdots x_{i_m})$ is the elementary symmetric function.

\begin{theorem}\label{s_comp}
Let $\mu$ and $\nu$ be any two compositions of length $l$ with codes $\alpha$ and $\beta$.  Then $s_\mu = \pm s_\nu$ if and only if $\alpha \sim \pm \beta$ as described in Sections \ref{sec:codereln} and \ref{sec:compos}.
\end{theorem}

\begin{proof}
Suppose that $s_\mu = \pm s_\nu$.  Then restricting to $x_{l+1} = x_{l+2} = \cdots = 0$ we obtain that $s_\mu(x_1, x_2, \ldots, x_l) = \pm s_\nu(x_1, x_2, \ldots, x_l)$.  Hence by Lemma \ref{SchurLem}, $\alpha \sim \pm \beta$.

Suppose that $\alpha \sim \pm \beta$.  Then by Lemma \ref{SchurLem},
$$s_\mu(x_1, x_2, \ldots, x_l) = \pm s_\nu(x_1, x_2, \ldots, x_l).$$
Now since the Schur functions $s_\mu$ and $s_\nu$ are both uniquely determined by the same Schur polynomial $s_\mu(x_1, x_2, \ldots, x_l) = \pm s_\nu(x_1, x_2, \ldots, x_l)$ (up to the sign), we have that $s_\mu = \pm s_\nu$.
\end{proof}

\begin{theorem}\label{genbern}
Given any composition $\mu = (\mu_1, \mu_2, \ldots, \mu_l)$, $$s_\mu = B_{\mu_1}B_{\mu_2} \cdots B_{\mu_l} \cdot 1.$$
\end{theorem}

\begin{proof}
If $B_\mu = 0$, then the columns of $({x_i}^{\mu_j+\delta_j})$ will be linearly dependent and $a_{\mu+\delta} = 0$, so $s_\mu = 0$.  Thus $B_\mu \cdot 1 = 0 = s_\mu$.

If $B_\mu \neq 0$, then by Corollary \ref{B_comp}, $B_\mu = \pm B_\lambda$ for some partition $\lambda$.  Hence $B_\mu \cdot 1 = \pm B_\lambda \cdot 1$.  By equation (\ref{bernschur}) we know that $B_\lambda \cdot 1 = s_\lambda$.  Finally, $s_\lambda = \pm s_\mu$ by Theorem \ref{s_comp}, where the sign is the same as above, since both come from the relation between the codes of $\mu$ and $\lambda$.  Therefore $B_\mu \cdot 1 = \pm B_\lambda \cdot 1 = \pm s_\lambda = s_\mu$.
\end{proof}

This theorem tells us that when the Bernstein operators act in an arbitrary order on 1, i.e. when they are indexed by a composition, then the result is the Schur function indexed by that same composition.  This generalizes Berstein's theorem (\ref{bernschur}), which gives the same result when the Bernstein operators act in nonincreasing order on 1, i.e. when they are indexed by a partition, obtaining the classical Schur functions (indexed by partitions) as the result.

\section{Schur $Q$-functions}\label{sec:schur}

We now turn our attention to Schur $Q$-functions and show some analogous results using codes of \textit{strict} partitions.  The Schur $Q$-functions are denoted by $Q_\lambda$, where $\lambda$ is a \textit{strict partition}, i.e. $\lambda_1 > \lambda_2 > \cdots > \lambda_l$.  The functions $Q_\lambda$ where $\lambda$ is a strict partition are an orthogonal basis of $\Lambda^-$, the ring of symmetric functions generated by the odd degree power sums $p_{2k+1}$.  Then from \cite{j1} we have that the twisted vertex operator $Y(t)$ given by
$$Y(t)=\sum_{n\in \mathbb{Z}}Y_nt^{-n}=exp\left(\sum_{k\geq1}\frac{2t^{-2k+1}}{k}p_{2k-1}\right)
exp\left(-\sum_{k\geq1}t^{2k-1}\frac{\partial}{\partial
p_{2k-1}}\right),
$$
which acts on $\Lambda^-$, satisfies the following two results.  The first is that $Y_{-\lambda}$ generates $Q_\lambda$ in the same way that the Bernstein operator $B_\lambda$ generates the Schur function $s_\lambda$.  That is
\begin{equation}\label{Qlambd}
Q_\lambda = Y_{-\lambda_1} Y_{-\lambda_2} \cdots Y_{-\lambda_l} \cdot 1,
\end{equation}
where $\lambda = (\lambda_1, \lambda_2, \ldots, \lambda_l)$ is a strict partition and $Q_\lambda$ is the Schur $Q$-function indexed by $\lambda$.  The second result is that the operators $Y_n$ anticommute.
\begin{equation}\label{Qcomm}
Y_nY_m = - Y_mY_n
\end{equation}
for any integers $m$ and $n$.  In particular, this means that $Y_nY_n = - Y_nY_n$, so $Y_nY_n = Y_n^2 = 0.$

Now we use the relationship among the $Y_n$'s to define a new equivalence relation on the set of signed codes of compositions.  We define $\alpha \sim \pm \beta$ if and only if $Y_{-\mu} = \pm Y_{-\nu},$ and $\alpha \sim 0$ if and only if $Y_{-\mu}=0,$ where $\alpha$ is the code of the composition $\mu$ and $\beta$ is the code of the composition $\nu$.  As in Section \ref{sec:codereln}, this is an equivalence relation since there is a one to one correspondence between a composition $\mu$ and its code $\alpha$.  Throughout this section we will refer only to this new relation.

Since the operators $Y_n$ anticommute, we know that $Y_\mu=0$ whenever $\mu$ contains any repetitions, so we can restrict ourselves to compositions $\mu$ with distinct components, but we want to recapture this in terms of codes alone.

\begin{proposition}\label{qcodeprop}
For any positive integer $k$, we have
\begin{eqnarray}
R\left(L^kUR^k\right) &\sim& \left(L^{k-1}UR^{k-1}\right)R\label{qcodereln1}\\
U\left(L^kUR^k\right) &\sim& -\left(L^kUR^k\right)U\label{qcodereln2}.
\end{eqnarray}
\end{proposition}

The proof of this proposition is similar to Proposition \ref{codeprop}.  Relation (\ref{qcodereln2}) follows from the commutation identity (\ref{Qcomm}) and corresponds to the code given by the altered path below.
\begin{center}
{\Yvcentermath1 $
\setlength{\unitlength}{.8cm}
\begin{picture}(6,1.5) 
\linethickness{0.075mm}
\put(0,1.1){\makebox(3,1)[b]{$n$}}
\put(0,-2.2){\makebox(6,1)[t]{$m$}}
\put(3,.15){\makebox(3,1)[b]{$k$}}
\put(0,0){\makebox(3,1){$\cdots$}}
\put(0,-1){\makebox(3,1){$\cdots$}}
\put(3,-1){\makebox(3,1){$\cdots$}}
\put(0,1){\line(1,0){3}}
\put(0,0){\line(1,0){6}}
\put(0,-1){\line(1,0){6}}
\put(0,1){\line(0,-1){2}}
\put(3,1){\line(0,-1){2}}
\put(6,0){\line(0,-1){1}}
\linethickness{.5mm}
\put(3,1){\line(1,0){3}}
\put(3,1){\line(0,-1){1}}
\put(3,0){\line(1,0){3}}
\put(6,0){\line(0,-1){1}}
\put(6,1){\circle*{.18}}
\put(6,-1){\circle*{.18}}
\end{picture}
\qquad
\rightarrow
\qquad
\begin{picture}(6,1.5)
\linethickness{0.075mm}
\put(0,1.1){\makebox(6,1)[b]{$m$}}
\put(0,-2.1){\makebox(3,1)[t]{$n$}}
\put(3,-1.1){\makebox(3,1)[t]{$k$}}
\put(0,0){\makebox(3,1){$\cdots$}}
\put(0,-1){\makebox(3,1){$\cdots$}}
\put(3,0){\makebox(3,1){$\cdots$}}
\put(0,1){\line(1,0){6}}
\put(0,0){\line(1,0){6}}
\put(0,-1){\line(1,0){3}}
\put(0,1){\line(0,-1){2}}
\put(3,1){\line(0,-1){2}}
\put(6,1){\line(0,-1){1}}
\linethickness{.5mm}
\put(6,1){\line(0,-1){1}}
\put(3,0){\line(1,0){3}}
\put(3,0){\line(0,-1){1}}
\put(3,-1){\line(1,0){3}}
\put(6,1){\circle*{.18}}
\put(6,-1){\circle*{.18}}
\end{picture}
$}
\end{center}
\vspace{1cm}

Note that unlike Proposition \ref{codeprop}, Proposition \ref{qcodeprop} says that in this setting we only decrease the index $k$ when we permute $\left(L^kUR^k\right)$ past an $R$.  However, like the previous case, the sign only changes when we permute $\left(L^kUR^k\right)$ past a $U$.


\begin{lemma}\label{QSameSum}
Suppose that the codes $\alpha$ and $\beta$ of two compositions $\mu$ and $\nu$ differ only by one of the relations (\ref{qcodereln1}) or (\ref{qcodereln2}).  Then $\mu$ and $\nu$ have the same number of components, $l$, and the same sum, $\mu_1+\mu_2+\cdots +\mu_l = \nu_1+\nu_2+\cdots +\nu_l$.
\end{lemma}

The proof is identical to Lemma \ref{SameSum}.

\begin{theorem}\label{BetaK+J}
Let $\mu$ be any compositions of $m$ with code $\alpha$.  Suppose that $\alpha$ can be written in the form
$$\alpha = \ldots \beta_3\beta_2\beta_1L^kU\gamma_1\gamma_2\gamma_3 \ldots$$
where $\beta = \ldots\beta_3\beta_2\beta_1$ consists only of $R$'s and $U$'s and $\beta_1 = U$.  Let $j$ be the smallest integer such that $\beta_{k+j}\ldots \beta_2\beta_1$ has $k$ $R$'s.  Then $Y_{-\mu} = 0$ if $\beta_{k+j+1}=U$ and $Y_{-\mu} = (-1)^jY_{-\nu}$ if $\beta_{k+j+1}=R$, where $\nu$ is the composition of $m$ with code given by
$$\ldots \beta_{k+j+1}U\beta_{k+j}\ldots \beta_2\beta_1L^k \gamma_1\gamma_2\ldots$$
\end{theorem}

\begin{proof}
By the minimality of $j$, we have that $\beta_{k+j} = R$.  By applying Proposition \ref{qcodeprop} $k+j$ times we have that
\begin{eqnarray*}
\alpha &=& \ldots \beta_3\beta_2\beta_1L^kU\gamma_1\gamma_2\gamma_3 \ldots \\
&\sim& \ldots \beta_3\beta_2\beta_1\left(L^kUR^k\right)L^k\gamma_1\gamma_2\gamma_3 \ldots \\
&\sim& (-1)^j \ldots \beta_{k+j+1}\left(L^0UR^0\right)\beta_{k+j}\ldots \beta_2\beta_1L^k\gamma_1\gamma_2\gamma_3 \ldots \\
&=& (-1)^j \ldots \beta_{k+j+1}U\beta_{k+j}\ldots \beta_2\beta_1L^k\gamma_1\gamma_2\gamma_3 \ldots
\end{eqnarray*}
since $j$ is the number of $U$'s in $\beta_{k+j}\ldots \beta_2\beta_1$.  If $\beta_{k+j+1} = U$, then the above code contains the subword $\beta_{k+j+1}U = UU \sim 0$, so $\alpha\sim 0$, so $Y_{-\mu} = 0$.  If $\beta_{k+j+1} = R$, then we have that $Y_{-\mu} = (-1)^j Y_{-\nu}$ for the partition $\nu$ which satisfies the conditions of the theorem.  In particular, $\nu$ is also a partition of $m$ by Lemma \ref{QSameSum}.
\end{proof}

\begin{corollary}\label{QCGLem}
Let $\lambda = (\lambda_1, \lambda_2, \ldots, \lambda_l)$ be a strict partition and $n$ be any integer with $n < \lambda_1$.  Let $k = \lambda_1 - n$, and let $\zeta$ be the letter immediately left of the $k^{\text{th}}$ $R$ left of the rightmost $U$ in the code of $\lambda$.
\begin{itemize}
\item If $\zeta=U$, then $Y_{-n}Y_{-\lambda} = 0$.
\item If $\zeta = R$, then $Y_{-n}Y_{-\lambda} = (-1)^{j+1} Y_{-\nu}$, where $j$ is the number of $U$'s between the rightmost $U$ and $\zeta$, and $\nu$ is the strict partition whose code is the code of $\lambda$ with $U$ inserted after $\zeta$.
\end{itemize}
\end{corollary}

This corollary is the analog of Corollary \ref{CGLem}, and similarly follows from Theorem \ref{BetaK+J} since the code $\alpha$ of $\lambda$ can be written in the form $\alpha = \beta\gamma$ in the notation of the theorem.

\begin{corollary}\label{Y_comp}
Let $\mu$ be any composition of $m$ with code $\alpha$.  Then either $Y_{-\mu} = 0$ or $Y_{-\mu} = \pm Y_{-\lambda}$ for some strict partition $\lambda$ of $m$ with the same length as $\mu$.
\end{corollary}

The proof of this corollary is the same as the proof presented for Corollary \ref{B_comp}, that is by induction on the number of $U$'s right of the leftmost $L$ in $\alpha$.  In this case though we can not use induction on the number of $L$'s in $\alpha$ since the number of $L$'s in the code do not decrease when we apply Theorem \ref{BetaK+J}, unlike Theorem \ref{BetaK}.

Similarly to Corollary \ref{B_comp}, this corollary provides a simple way to compute $Y_{-\mu} \cdot 1$ in terms of Schur $Q$-functions for any composition $\mu$, since $Y_{-\lambda}\cdot 1 = Q_\lambda$ by equation (\ref{Qlambd}) for any strict partition $\lambda$.

In fact, given any composition $\mu$, the number of times we have to apply Theorem \ref{BetaK+J} to get  $Y_{-\mu} = \pm Y_{-\lambda}$ for some strict partition $\lambda$ is less than or equal to the number of $U$'s right of the leftmost $L$ in the code of $\mu$, which is the largest $i$ such that $\mu_i<\mu_{i+1}$, i.e. the position of the last exceedance in $\mu$.

The above corollary follows from the previous theorem, but since in this case we know that the $Y_{-n}$ anticommute, we actually have the following stronger statement.

\begin{proposition}\label{sigmaprop}
Let $\mu$ be any composition of $m$ with length $l$.  Then either $Y_{-\mu} = 0$ or $Y_{-\mu} = \sgn(\sigma) Y_{-\sigma(\mu)}$ for any permutation $\sigma \in S_l$.
\end{proposition}

\begin{proof}
The proof of this statement is immediate from (\ref{Qcomm}).  In particular, $Y_{-\mu} = 0$ exactly when $\mu$ has a repeated index.  The second case, $Y_{-\mu} = \sgn(\sigma) Y_{-\sigma(\mu)}$ follows from the fact that any permutation $\sigma$ can be written as a sequence of adjacent transpositions $\sigma = \sigma_1 \sigma_2 \cdots \sigma_k$, and $\sgn(\sigma) = (-1)^k$.  Then $Y_{-\mu} = - Y_{-\sigma_k(\mu)} = + Y_{-\sigma_{k-1}\sigma_k(\mu)} = \cdots = (-1)^k Y_{-\sigma_1 \sigma_2 \cdots \sigma_k (\mu)} = \sgn(\sigma) Y_{-\sigma(\mu)}$.
\end{proof}

This proposition follows immediately from known results.  We include it only to show that Corollary \ref{Y_comp} gives an only slightly less general version of this result using only codes.

\begin{definition}\label{[i]1}
For any strict partition $\lambda$, define $\lambda^{[i]}$ to be the strict partition obtained from the code of $\lambda$ by inserting a $U$ between the $i^{\text{th}}$ pair of consecutive $R$'s.
\end{definition}

In particular, this means that $\lambda^{[i]}$ is the strict partition with the $i^{\text{th}}$ smallest positive integer not already in $\lambda$ inserted into $\lambda$.

\begin{theorem}\label{QThm}
For any strict partition $\lambda = (\lambda_1, \lambda_2, \ldots, \lambda_l)$,
\begin{eqnarray}
Y(t)Q_\lambda &=& \sum_{j=0}^{l}\, \sum_{n=\lambda_{j+1}+1}^{\lambda_{j}-1} \,
(-1)^j \,
t^n \,
Q_{(\lambda_1, \lambda_2, \ldots, \lambda_j, n, \lambda_{j+1} \ldots, \lambda_l)}\label{QThmA} \\
Y(t)Q_\lambda &=& \sum_{i\geq 0} \,
(-1)^{l + |\lambda| - |\lambda^{[i]}| + i} \,
t^{|\lambda^{[i]}| - |\lambda|} \,
Q_{\lambda^{[i]}} \label{QThmB}
\end{eqnarray}
where we take the convention $\lambda_0=\infty$ and $\lambda_{l+1} = -1$.
\end{theorem}

\begin{proof}
From the definition of $Y(t) = \sum_{n\in \mathbb{Z}} Y_nt^{-n}$, and equation (\ref{Qlambd}), we have that $Y(t)Q_\lambda = \sum_{n\in\mathbb{Z}}Y_nt^{-n}(Y_{-\lambda_1}Y_{-\lambda_2}\cdots Y_{-\lambda_l}\cdot 1)$.  So it is sufficient to consider the coefficient of $t^{-n}$, $Y_{-n}Y_{-\lambda_1}Y_{-\lambda_2}\cdots Y_{-\lambda_l}$.

If $n > \lambda_1$, then $Y_{-n}Y_{-\lambda} = Y_{-(n,\lambda)}$ is already in decreasing order.  In fact, since $(n, \lambda) = \lambda^{(n-l)}$, we can write $Y_{-n}Y_{-\lambda} \cdot 1= Y_{-(n,\lambda)} \cdot 1 = Y_{\lambda^{(n-l)}} \cdot 1 = Q_{\lambda^{(n-l)}}$.

If $n \leq \lambda_1$, then we write the code of $\lambda$ in the form
$$\alpha = \ldots \beta_3\beta_2\beta_1RRR\ldots$$
where $\beta_1 = U$.  Let $j$ be the smallest integer such that $\beta_{k+j}\ldots \beta_2\beta_1$ has $k$ $R$'s, where $k = \lambda_1 - n$.

From Corollary \ref{QCGLem}, if $\beta_{k+j+1} = U$, then $\alpha \sim 0$, so $Y_{-n}Y_{-\lambda} = 0$.  Since there are $j$ $U$'s in $\beta_{k+j}\ldots \beta_2\beta_1$, $\beta_{k+j+1}$ is the $(j+1)^{\text{th}}$ $U$ from the right in $\alpha$.  This naturally divides the $r_1(\lambda)$ $R$'s left of $\beta_1$, the rightmost $U$ in $\alpha$, into $r_{j+1}(\lambda)$ $R$'s left of $\beta_{k+j+1} = U$ and $k$ $R$'s right of $\beta_{k+j+1}$.  Then we have that $n = \lambda_1 - k = r_1(\lambda) - k = \left( r_{j+1}(\lambda) + k \right) - k = r_{j+1}(\lambda) = \lambda_{j+1}$.  Thus $Y_{-n}Y_{-\lambda} = 0$ exactly when $n=\lambda_i$ for some $i$, that is, when $n$ is already a component of $\lambda$, as we expect from (\ref{Qlambd}) and (\ref{Qcomm}).

If $n\neq \lambda_i$ for any $i$, then from Corollary \ref{QCGLem}, $Y_{-n}Y_{-\lambda} = (-1)^j Y_{-\nu}$, where $\nu$ is the strict partition with code
$$\alpha' = \ldots \beta_{k+j+1}U\beta_{k+j} \ldots \beta_2\beta_1RRR\ldots$$
Comparing this to the code $\alpha$ of $\lambda$, we have that $r_{j+1}(\nu) = n = \lambda_1 - k$, $r_i(\lambda) = r_i(\lambda)$ for all $i \leq j$, and $r_{i+1}(\nu) = r_i(\lambda)$ for all $i > j$.  Thus $\nu = (\lambda_1, \lambda_2, \ldots, \lambda_j, n, \lambda_{j+1}, \ldots, \lambda_l)$, and thus (\ref{QThmA}) holds.

There are $n$ $R$'s left of the inserted $U$ in $\alpha'$.  Each $U$ corresponding to a component of $\nu$ will be immediately after an $R$ ($RU$), since $\nu$ is a strict partition.  Thus $l-j$ of the $R$'s left of the inserted $U$ will be immediately before a $U$, so $n - (l-j) = n-l+j$ of the $R$'s left of the inserted $U$ will be immediately before an $R$.  This includes $\beta_{k+j+1}$, the $R$ immediately before the inserted $U$, so this $U$ is inserted into the $(n-l+j)^{\text{th}}$ position from the left.  So $\nu = \lambda_1, \lambda_2, \ldots, \lambda_j, n, \lambda_{j+1}, \ldots, \lambda_l) = \lambda^{[n-l+j]}$.

Let $i = n-l+j$.  Then $\lambda^{[i]} = \nu$, $|\lambda^{[i]}| - |\lambda| = n$, and $l+|\lambda| - |\lambda^{[i]}| + i = l - n + (n-l+j) = j$.  Finally, notice that as $n$ runs over all summands which contribute to (\ref{QThmA}), $i$ will run over all nonnegative integers ($i=0$ corresponds to the case $n=0$), so (\ref{QThmB}) holds.
\end{proof}

In light of the results in this section, particularly Proposition \ref{sigmaprop} and Theorem \ref{QThm}, one might ask if there is an intuitive way to define a function $Q_\mu$ indexed by compositions such that $Q_\lambda$ is the Schur $Q$-function when $\lambda$ is a partition which satisfies analogous statements to Theorem \ref{s_comp} and Theorem \ref{genbern}.  That is, can we generalize Schur $Q$-functions as we generalized Schur functions in Section \ref{sec:genschur}?

Unfortunately the definition of the Schur $Q$-functions does not lend itself to generalization in this manner as the definition of the Schur functions we used in Section \ref{sec:genschur} did.  However, even if we could generalize the definition of the Schur $Q$-functions in an intuitive way, the analogous result to Theorem \ref{s_comp} would tell us that $Q_\mu = \sgn(\sigma) Q_{\sigma(\mu)}$ for any permutation $\sigma\in S_l$ by Proposition \ref{sigmaprop}.  This implies that $Q_\mu = Q_\lambda$ exactly when $\mu$ is a rearrangement of $\lambda$.  In other words, the result we would get from generalizing Schur $Q$-functions (unlike when we generalized Schur functions) would be trivial.

If we choose as our definition $Q_\mu = Y_{-\mu} \cdot 1$ for any composition $\mu$, then we recover the two results mentioned in the previous paragraph immediately.  Namely that $Q_\mu = \sgn(\sigma) Q_{\sigma(\mu)}$ for any permutation $\sigma\in S_l$ (by Proposition \ref{sigmaprop}) and $Q_\mu = Q_\lambda$ exactly when $\mu$ is a rearrangement of $\lambda$.  The drawback here is that we {\it define} $Q_\mu$ to satisfy the same defining relation as the Schur $Q$-functions, whereas in Section \ref{sec:genschur} we were able to generalize Schur functions and {\it prove} the relationship between the new function $s_\mu$ and $B_\mu \cdot 1$ in order to better understand the latter.  So this definition, while valid, is not terribly useful or illustrative.

\section{Shifted Codes}\label{sec:shift}

In this section we will relate codes to {\it shifted codes} of strict partitions and use these to study the analog of the Bernstein operators for Schur $Q$-functions.

\begin{definition}
Define the \textit{shifted code} of a strict partition $\lambda$ to be the infinite sequence of letters $R$ and $U$ obtained by tracing along the bottom-right edge of the shifted Young diagram of shape $\lambda$ in the fourth quadrant together with the positive $x$-axis, starting at the bottom-right corner of the leftmost box on the bottom row of the diagram.
\end{definition}

\begin{example}\label{ShiftEx1}
For the strict partition $\lambda = (4,2,1)$, the path described above is shown in bold below.
\begin{center}
\setlength{\unitlength}{.8cm}
\begin{picture}(6,3) 
\linethickness{0.075mm}
\put(0,3){\line(1,0){6}}
\put(0,2){\line(1,0){4}}
\put(1,1){\line(1,0){2}}
\put(2,0){\line(1,0){1}}
\put(0,3){\line(0,-1){1}}
\put(1,3){\line(0,-1){2}}
\put(2,3){\line(0,-1){3}}
\put(3,3){\line(0,-1){3}}
\put(4,3){\line(0,-1){1}}
\linethickness{.5mm}
\put(3,2){\line(0,-1){2}}
\put(3,2){\line(1,0){1}}
\put(4,3){\line(0,-1){1}}
\put(4,3){\line(1,0){2}}
\put(3,0){\circle*{.18}}
\end{picture}
\end{center}
Then the shifted code of $\lambda$ is given by  $\alpha = UURURRR\ldots$
\end{example}


Using the tools we developed to study codes of compositions, we will now show how the shifted code of a strict partition can be obtained directly from the code of that partition.

\begin{definition}
Given a strict partition $\lambda$ with code $\alpha$, replace each $U$ in $\alpha$ with $UL$ and use the fact that $LR \sim \phi$ to cancel wherever possible.  Call the resulting sequence of letters $R$, $L$, and $U$ the \textit{preshifted code}
of $\lambda$.
\end{definition}

If the infinite prefix ``$\ldots ULULU$" is removed from the preshifted code of a strict partition, then you obtain exactly the shifted code of that partition.  The reason for this is that replacing each $U$ with $UL$ makes the diagram left aligned along the line $y = - x$ rather than the negative $y$-axis.

\begin{example}
For the empty partition $\lambda = \phi$, the code of $\lambda$ is $\ldots UUURRR\ldots$ from the diagram:
\begin{center}
\setlength{\unitlength}{.8cm}
\begin{picture}(2,2) 
\linethickness{.5mm}
\put(0,2){\line(0,-1){2}}
\put(0,2){\line(1,0){2}}
\end{picture}
\end{center}
The preshifted code is $\ldots ULULURRR\ldots$ from the diagram:
\begin{center}
\setlength{\unitlength}{.8cm}
\begin{picture}(4,4) 
\linethickness{.5mm}
\multiput(2.5,1)(.35,-.35){3}{\circle*{.1}}
\put(0,4){\line(1,0){4}}
\put(0,4){\line(0,-1){1}}
\put(0,3){\line(1,0){1}}
\put(1,3){\line(0,-1){1}}
\put(1,2){\line(1,0){1}}
\put(2,2){\line(0,-1){1}}
\end{picture}
\end{center}
And the shifted code is $RRR\ldots$ from the diagram:
\begin{center}
\setlength{\unitlength}{.8cm}
\begin{picture}(3,0) 
\linethickness{.5mm}
\multiput(2.35,0)(.35,0){3}{\circle*{.1}}
\put(0,0){\line(1,0){2}}
\put(0,0){\circle*{.18}}
\end{picture}
\end{center}
\end{example}

Shifted codes can provide an alternative (but equivalent) definition of $\lambda^{[i]}$ to the one given in Definition
\ref{[i]1}.

\begin{definition}
For any strict partition $\lambda$, define $\lambda^{[i]}$ to be the strict partition obtained from the shifted code of $\lambda$ by replacing the $i^{\text{th}}$ $R$ from the left in the shifted code of $\lambda$ by $U$.
\end{definition}

Notice the similarity to Definition \ref{(i)}, which defined $\lambda^{(i)}$ the exact same way using codes of partitions rather than shifted codes of strict partitions.

\begin{definition}
Given a composition $\mu$ with code $\alpha$, replace each $U$ in $\alpha$ with $UL$.  Call the resulting sequence of letters $R$, $L$, and $U$ the \textit{preshifted code} of $\mu$.  Remove the prefix ``$\ldots ULULU$" to obtain the \textit{shifted code} of $\mu$.
\end{definition}

\begin{example}
For the composition $\mu = (2,3,1)$ the shifted code is obtained from the path shown below in bold.
\begin{center}
\setlength{\unitlength}{.8cm}
\begin{picture}(6,3) 
\linethickness{0.075mm}
\put(0,3){\line(1,0){6}}
\put(0,2){\line(1,0){4}}
\put(1,1){\line(1,0){3}}
\put(2,0){\line(1,0){1}}
\put(0,3){\line(0,-1){1}}
\put(1,3){\line(0,-1){2}}
\put(2,3){\line(0,-1){3}}
\put(3,2){\line(0,-1){2}}
\put(4,2){\line(0,-1){1}}
\linethickness{.5mm}
\put(3,1){\line(0,-1){1}}
\put(3,1){\line(1,0){1}}
\put(4,2){\line(0,-1){1}}
\put(2,2){\line(1,0){2}}
\put(2,3){\line(0,-1){1}}
\put(2,3){\line(1,0){4}}
\put(3,0){\circle*{.18}}
\end{picture}
\end{center}
Then the shifted code of $\mu$ is given by $\alpha = URULLURRR\ldots$ and the preshifted code of $\mu$ is $\ldots ULULUURULLURRR\ldots$
\end{example}

We can now use the relationship among the $Y_n$'s to define yet another equivalence relation, this one on the set of signed shifted codes of compositions.  We define $\alpha \sim \pm \beta$ if and only if $Y_{-\mu} = \pm Y_{-\nu}$, and $\alpha \sim 0$ if and only if $Y_{-\mu} = 0$, where $\alpha$ is the shifted code of the composition $\mu$ and $\beta$ is the code of the composition $\nu$.  As before, this will be an equivalence relation since there is a one to one correspondence between a composition $\mu$ and its shifted code $\alpha$.

\begin{proposition}\label{shcodeprop}
For any positive integer $k$, we have
\begin{eqnarray}
R\left(L^{k+1}UR^k\right) &\sim& \left(L^kUR^k\right)R\label{shcodereln1}\\
U\left(L^{k+1}UR^k\right) &\sim& -\left(L^kUR^k\right)U.\label{shcodereln2}
\end{eqnarray}
\end{proposition}

Again the proof is similar to Proposition \ref{codeprop}.  Relation \ref{shcodereln2} follows from the commutation identity (\ref{Qcomm}) and corresponds to the shifted code given by the altered path below, where $k = m - n$ as before.
\begin{center}
{\Yvcentermath1 $
\setlength{\unitlength}{.8cm}
\begin{picture}(7,1.5) 
\linethickness{0.075mm}
\put(0,1.1){\makebox(3,1)[b]{$n$}}
\put(1,-2.2){\makebox(6,1)[t]{$m$}}
\put(3,.15){\makebox(4,1)[b]{$k+1$}}
\put(3,1.1){\makebox(3,1)[b]{$k$}}
\put(0,0){\makebox(3,1){$\cdots$}}
\put(1,-1){\makebox(3,1){$\cdots$}}
\put(4,-1){\makebox(3,1){$\cdots$}}
\put(0,1){\line(1,0){3}}
\put(0,0){\line(1,0){7}}
\put(1,-1){\line(1,0){6}}
\put(0,1){\line(0,-1){1}}
\put(1,0){\line(0,-1){1}}
\put(3,1){\line(0,-1){1}}
\put(4,0){\line(0,-1){1}}
\put(7,0){\line(0,-1){1}}
\linethickness{.5mm}
\put(3,1){\line(1,0){3}}
\put(3,1){\line(0,-1){1}}
\put(3,0){\line(1,0){4}}
\put(7,0){\line(0,-1){1}}
\put(6,1){\circle*{.18}}
\put(7,-1){\circle*{.18}}
\end{picture}
\qquad
\rightarrow
\qquad
\begin{picture}(7,1.5)
\linethickness{0.075mm}
\put(0,1.1){\makebox(6,1)[b]{$m$}}
\put(1,-2.1){\makebox(3,1)[t]{$n$}}
\put(4,-1.1){\makebox(2,1)[t]{$k-1$}}
\put(4,-2.15){\makebox(3,1)[t]{$k$}}
\put(0,0){\makebox(3,1){$\cdots$}}
\put(1,-1){\makebox(3,1){$\cdots$}}
\put(3,0){\makebox(3,1){$\cdots$}}
\put(0,1){\line(1,0){6}}
\put(0,0){\line(1,0){6}}
\put(1,-1){\line(1,0){3}}
\put(0,1){\line(0,-1){1}}
\put(1,0){\line(0,-1){1}}
\put(3,1){\line(0,-1){1}}
\put(4,0){\line(0,-1){1}}
\put(6,1){\line(0,-1){1}}
\multiput(7,1)(0,-.2){10}{\line(0,-1){.1}} 
\multiput(7,1)(-.2,0){5}{\line(-1,0){.1}} 
\multiput(7,0)(-.2,0){5}{\line(-1,0){.1}} 
\multiput(6,0)(0,-.2){5}{\line(0,-1){.1}} 
\linethickness{.5mm}
\put(6,1){\line(0,-1){1}}
\put(4,0){\line(1,0){2}}
\put(4,0){\line(0,-1){1}}
\put(4,-1){\line(1,0){3}}
\put(6,1){\circle*{.18}}
\put(7,-1){\circle*{.18}}
\end{picture}
$}
\end{center}
\vspace{1cm}

Notice that the relation on shifted codes of compositions in Proposition \ref{shcodeprop} is identical to the relation on codes of compositions in Proposition \ref{codeprop}.  The two propositions give the same relation;  the only difference being that the index $k$ in Proposition \ref{codeprop} has been replaced by $k+1$ in Proposition \ref{shcodeprop} to preserve the identity $k = m - n$.

With this identification we can prove the shifted code analog of each result in Section \ref{sec:schur} in exactly the same way as each corresponding result in Section \ref{sec:compos}.  We include the statements of these results for completeness.


\begin{lemma}\label{shSameSum}
Suppose that the codes $\alpha$ and $\beta$ of two compositions $\mu$ and $\nu$ differ only by one of the relations (\ref{shcodereln1}) or (\ref{shcodereln2}).  Then $\mu$ and $\nu$ have the same number of components, $l$, and the same sum, $\mu_1+\mu_2+\cdots +\mu_l = \nu_1+\nu_2+\cdots +\nu_l$.
\end{lemma}

The proof is identical to Lemma \ref{SameSum}.

\begin{theorem}\label{shBetaK}
Let $\mu$ be any composition of $m$ with shifted code $\alpha$.  Suppose that $\alpha$ can be written in the form
$$\alpha = \beta_t \ldots \beta_3\beta_2\beta_1L^kU \gamma_1 \gamma_2 \gamma_3 \ldots$$
where $\beta = \beta_t \ldots \beta_3\beta_2\beta_1$ consists only of $R$'s and $U$'s and $\beta_1=U$.
\begin{itemize}
\item If $\beta_k = U$, then $Y_{-\mu} = 0$.
\item If $\beta_k = R$, then $Y_{-\mu} = (-1)^j Y_{-\nu}$, where $j$ is the number of $U$'s in $\beta_{k-1}\ldots \beta_2\beta_1$ and $\nu$ is the composition of $m$ with shifted code given by
$$\beta_t \ldots \beta_{k+1}U\beta_{k-1}\ldots \beta_3\beta_2\beta_1L^{k-1} \gamma_1 \gamma_2 \gamma_3 \ldots$$
\end{itemize}
\end{theorem}

The proof is identical to Theorem \ref{BetaK}.

\begin{corollary}\label{shCGLem}
Let $\lambda = (\lambda_1, \lambda_2, \ldots, \lambda_l)$ be a strict partition and $n$ be any integer with $n < \lambda_1$.  Let $k = \lambda_1 - n$, and let $\zeta$ be the letter $k-1$ positions left of the rightmost $U$ in the shifted code of $\lambda$.
\begin{itemize}
\item If $\zeta=U$, then $Y_{-n}Y_{-\lambda} = 0$.
\item If $\zeta = R$, then $Y_{-n}Y_{-\lambda} = (-1)^{j+1} Y_{-\nu}$, where $j$ is the number of $U$'s between the rightmost $U$ and $\zeta$, and $\nu$ is the partition obtained by replacing $\zeta$ by $U$.
\end{itemize}
\end{corollary}

The proof is identical to Corollary \ref{CGLem}.  It follows directly from Theorem \ref{shBetaK} for the special case where $\mu = (n, \lambda_1, \lambda_2, \ldots, \lambda_l)$ and $\gamma = RRR\ldots$  If $n = \lambda_j$ for some $j$, then we would have $\zeta = U$ and $Y_{-\mu} = 0$.  So the case $\zeta = R$ corresponds to the \textit{strict} partition $\nu = (\lambda_1, \lambda_2, \ldots, \lambda_j, n, \lambda_{j+1}, \ldots, \lambda_l)$.

\begin{corollary}\label{shY_comp}
Let $\mu$ be any composition of $m$ with shifted code $\alpha$.  Then either $Y_{-\mu} = 0$ or $Y_{-\mu} = \pm Y_{-\lambda}$ for some strict partition $\lambda$ of $m$ with the same length as $\mu$.
\end{corollary}

The proof is identical to Corollary \ref{B_comp}.  As in that case, we can prove the result using induction on either the number of $L$'s in $\alpha$ or on the number of $U$'s to the right of the leftmost $L$ in $\alpha$.  Notice that Corollary \ref{Y_comp} gives the exact same result as this corollary, however in that case only the latter method of proof is intuitive.  This is one example of the strength of using shifted codes to study this problem.

\begin{theorem}\label{shQThm}
For any strict partition $\lambda = (\lambda_1, \lambda_2, \ldots, \lambda_l)$,
\begin{eqnarray}
Y(t)Q_\lambda &=& \sum_{j=0}^{l}\, \sum_{n=\lambda_{j+1}+1}^{\lambda_{j}-1} \,
(-1)^j \,
t^n \,
Q_{(\lambda_1, \lambda_2, \ldots, \lambda_j, n, \lambda_{j+1} \ldots, \lambda_l)}\label{QThmA} \\
Y(t)Q_\lambda &=& \sum_{i\geq 0} \,
(-1)^{l + |\lambda| - |\lambda^{[i]}| + i} \,
t^{|\lambda^{[i]}| - |\lambda|} \,
Q_{\lambda^{[i]}} \label{QThmB}
\end{eqnarray}
where we take the convention $\lambda_0=\infty$ and $\lambda_{l+1} = -1$.
\end{theorem}

The proof is almost identical to the proof of Theorem \ref{CGCodeThm}.  In the setting of shifted codes, we replace codes with shifted codes and $\lambda^{(i)}$ with $\lambda^{[i]}$.  This means that the number, $j$, of $U$'s in $\beta_k \ldots \beta_2 \beta_1$ will be $j = l + |\lambda| - |\lambda^{[i]}| + i$ rather than $j = |\lambda|-|\lambda^{(i)}|+i-1$, since $\lambda^{[i]} = (\lambda_1, \lambda_2, \ldots, \lambda_j, i+(l-j), \lambda_{j+1}, \ldots, \lambda_l)$ and $\lambda^{(i)} = (\lambda_1-1, \lambda_2-1, \ldots, \lambda_j-1, i-1, \lambda_{j+1}, \ldots, \lambda_l)$.  Finally, we must take the convention $\lambda_{l+1} = -1$ rather than zero in this case, since $Y_0$ represents the nontrivial insertion of zero into the strict partition, whereas $B_0$ acts as trivially on 1.

Notice that this gives the exact same result as Theorem \ref{QThm}, but the proof will follow that of the Schur function case.


Since the only difference between preshifted codes and shifted codes is the prefix $\ldots ULULU$, we can also state and prove each of these results in terms of preshifted codes, where we concern ourselves only with $L$'s not in the prefix $\ldots ULULU$.  The statements and proofs of these results will be otherwise identical to those using shifted codes.\\
\\
{\bf Acknowledgments:}
The first author would like to thank Sarah Mason and the attendees of the AMS Special Session on Algebraic and Geometric Combinatorics at Georgia Southern University (2011) for their insightful questions. The second author
gratefully acknowledges the partial support of Max-Planck Institut f\"ur Mathematik in Bonn, Simons Foundation grant 198129, and NSFC grant 10728102 during this work. The first two authors also acknowledge the partial support from NSF grants 1014554 and 1137837.

\end{document}